\documentclass[leqno,11pt]{amsart} 
\usepackage{amssymb}
\usepackage[english]{babel}

 \theoremstyle{plain} 
\newtheorem{theo}{\indent\sc Theorem}[section]
\newtheorem{lemm}[theo]{\indent\sc Lemma}
\newtheorem{cor}[theo]{\indent\sc Corollary}
\newtheorem{prop}[theo]{\indent\sc Proposition}
\newtheorem*{theo*}{Theorem}
\newtheorem*{cor*}{Corollary}
\theoremstyle{definition} 
\newtheorem{defi}[theo]{\indent\sc Definition}
\newtheorem{rem}[theo]{\indent\sc Remark}

\newcommand{\vna}{von Neumann algebra}

\newcommand{\sps}{standard probability space}
\newcommand{\eqre}{equivalence relation}
\newcommand{\bieqre}{bimodule equivalence relation}

\newcommand{\oreqre}{orbit equivalence relation}
\newcommand{\Hs}{Hilbert space}
\newcommand{\ifff}{{if and only if}}

\newcommand{\alev}{a.e.}

\newcommand{\mesp}{measure space}
\newcommand{\wrt}{with respect to}

\newcommand{\R}{\mathbb R}
\newcommand{\C}{\mathbb C}

\newcommand{\h}{\mathcal H}
\newcommand{\G}{\Gamma}
\newcommand{\LM}{L^2(M)}
\newcommand{\AM}{A\subset M}

\newcommand{\NGH}{N_G(H)}
\newcommand{\NMA}{N_M(A)}
\newcommand{\hst}{\h_{s,t}}
\newcommand{\K}{\mathcal K}
\newcommand{\xist}{\xi_{s,t}}

\newcommand{\Ymus}{{\mathcal Y}}
\newcommand{\RT}{\mathcal R}
\newcommand{\Y}{{\mathcal{Y}}}

\newcommand{\WR}{\mathcal {B}}
\newcommand{\Ncal}{\mathcal N}

\newcommand{\intI}{\int^\oplus_Y}
\newcommand{\intII}{\int^\oplus_{Y^2}}
\newcommand{\Li}{{L^\infty(Y,\nu)}}
\newcommand{\Lii}{{L^\infty(Y^2,\mu)}}

\newcommand{\Ldd}{L^2(Y^2,\mu)}
\newcommand{\CY}{{\mathcal C(Y)}}

\newcommand{\IY}{\mathfrak I(Y,\nu)}

\newcommand{\wH}{\widehat H}
\newcommand{\LMl}{L^2(M_l)}
\newcommand{\AMl}{A_l\subset M_l}
\newcommand{\Al}{\mathcal A_l}
\newcommand{\otimesl}{\otimes_{l\in\Lambda}}
\newcommand{\bigotimesl}{\bigotimes_{l\in\Lambda}}
\newcommand{\NMAl}{N_{M_l}(A_l)}
\begin{document}

\title[The Takesaki equivalence relation]{The Takesaki equivalence relation for maximal abelian subalgebras} 

\author[A. Brothier]{Arnaud Brothier$^*$} 

\thanks{$^{*}$Partially supported by the Region Ile de France and by ERC Starting Grant VNALG-200749.}
\address{
Institute of Mathematics of Jussieu \endgraf
Universtity Paris Diderot \endgraf
175 rue du Chevaleret \endgraf
Paris 75013 \endgraf
France
}
\email{brot@math.jussieu.fr}
\address{
K.U.Leuven \endgraf
Department of Mathematics\endgraf
Celestijnenlaan 200b - bus 2400\endgraf
BE-3001 Heverlee\endgraf
}
\email{arnaud.brothier@wis.kuleuven.be}
\begin{abstract}
For a maximal abelian subalgebra $A\subset M$ in a finite \vna, we consider an invariant due to Takesaki which is an equivalence relation on a standard probability space.
We give several characterizations of this invariant and show that it can be reconstructed from the $A$-bimodule structure of the GNS Hilbert space $\LM$.
In particular, we show that this invariant is induced by the action of the normalizer on $A$. Hence, this gives a new proof to a question of Takesaki.
\end{abstract}
\maketitle

\section*{Introduction}

In this paper, we study maximal abelian subalgebras (MASAs) in a finite \vna\ with a separable predual. 
We will always denote such an inclusion as $\AM$ and we fix a faithful normal unital trace $\tau$ on $M$.
The study of MASAs began with the work of Dixmier in 1954 \cite{Dixmier_anneaux_max_ab}. (See \cite{sinclair_smith_masas} for a general introduction on the subject.)
He considered the \textit{normalizer} $\NMA$ which is the group of unitaries $u\in M$ such that $uAu^*=A$.
In 1963, Takesaki introduced in \cite{takesaki_invariant_masa} a measure theoretical invariant for a MASA.
An explicit presentation of this invariant is given below.

Let us first define the \textit{Takesaki equivalence relation.}
We refer the reader to \cite[Chap.4, \S 8]{Takesaki_bouquin_I} for a presentation of the general theory of direct integrals of \Hs s, representations, and \vna s.
Let $Y$ be a compact Hausdorff space and $\nu$ a Borel probability measure on it such that $A$ is isomorphic to the \vna\  $L^\infty(Y,\nu)$.
We fix such an isomorphism and identify the two \vna s.
Let $\LM$ be the GNS Hilbert space associated to the trace $\tau$ and $x\mapsto x\Omega$ the embedding of $M$ in $\LM$.
Let $\pi,\rho$ be the left and right actions of $M$ on the Hilbert space $\LM$, i.e. $\pi(x)\rho(y)(z\Omega)=xzy\Omega$.
Consider the measurable field of \Hs s $\{\mathcal K_t,\,t\in Y\}$ such that $\LM$ is equal to the direct integral
$$\intI \mathcal K_t d\nu(t),$$
such that $\rho(A)$ becomes the algebra of all diagonalizable operators.
Let $B\subset M$ be a separable $C^*$-subalgebra that is dense for the weak topology.
Consider the measurable field of representations of $B$, $\{\pi_t,\,t\in Y\}$, such that
$$\pi\vert_B=\intI \pi_t d\nu(t),$$
where $\pi\vert_B$ denotes the restriction to $B$ of the standard representation.
\begin{defi}\label{defi_Tak_Tak_equirel}
Let $\RT$ be the equivalence relation on $Y$ defined by $(s,t)\in\RT$ if and only if the representation $\pi_s$ is unitarily equivalent to $\pi_t$.
It is the \textit{Takesaki equivalence relation}.
\end{defi}
We write $\pi_s\simeq \pi_t$ to say that the two representations are unitarily equivalent.
\begin{defi}\label{defi_approx}
Let $E,F\subset Y^2$ be some subsets, we say that $E$ is \textit{weakly contained} in $F$ if there exists a null set $N\subset Y$ such that $E\backslash N^2\subset F$, where $E\backslash N^2=\{x\in E,\,x\notin N^2\}$.
We denote this by $E\prec F$. 
This defines a partial order.
We say that $E$ is \textit{equivalent} to $F$ if $E\prec F$ and $F\prec E$ and denote it by $E\equiv F$.
This defines an equivalence relation on the subsets of $Y^2$. We denote the equivalence class of a subset $E$ by $\widehat E$. 
\end{defi}
\begin{defi}\label{defi_Tak_Tak_inv}
Let $\widehat{\mathcal R}$ be the equivalence class of $\mathcal R$ for $\equiv$.
It is an invariant for the MASA $\AM$ that we call the \textit{Takesaki invariant}.
In particular it does not depend of the choice of the $C^*$-algebra $B$, see \cite{takesaki_invariant_masa}.
We say that a MASA is \textit{Takesaki simple} if $\RT\equiv\Delta Y$, where $\Delta Y=\{(t,t),\,t\in Y\}$ is the diagonal of $Y$.
\end{defi}
Let us define an other \eqre. 
Consider the normaliser $\NMA$ and a countable subgroup $G<\NMA$ such that the bicommutant $\{G\cup A\}''$ is equal to $\NMA''$ inside $M$.
The group $G$ acts on $A$, hence this gives an action on the space $(Y,\nu)$.  
We denote by $\Ncal_G$ the orbit equivalence relation.
This \eqre\ does not depend of the choice of the group $G$ (see proposition \ref{prop_Ncal}), therefore we simply denote it by $\Ncal$. 

Takesaki proved in \cite[theorem 1.2]{takesaki_invariant_masa} that $\Ncal\prec \RT$. He asked if $\RT$ is a countable, quasi-invariant \eqre\ and  if  $\RT\equiv \Ncal$.
We recall that an \eqre\ is quasi-invariant if the saturation of a null set is still a null set.
In the mid 70's, Hahn developed a theory of measure groupoids \cite{Hahn_Haar_meas_I}, \cite{Hahn_Haar_meas_II}.
Using this theory, he proved in \cite{Hahn_reconstruction} that $\RT\equiv\Ncal$.

The author wants to indicate that all the results presented in this paper have been proved without knowing the work of Hahn.
We give here an elementary proof of the equivalence $\RT\equiv\Ncal$.
Furthermore, we define a third equivalence relation $\WR$  that we call the \textit{bimodule equivalence relation} and prove that $\RT\equiv\WR$.
In particular we show the surprising fact that the Takesaki \eqre\ can be reconstructed from the $A$-bimodule structure of $\LM$.

The \bieqre\ is defined in a similar way as $\RT$ except that we replace the separable weakly dense $C^*$-subalgebra $B\subset M$ by a separable weakly dense $C^*$-subalgebra $D\subset A$ (see definition \ref{defi_Tak_weak_Tak_eqre}).

One of the key arguments to prove this theorem is to show that the Takesaki \eqre\ is quasi-invariant.
We introduce a subset $\Y\subset Y\times Y$ that we call \textit{the set of atoms} (see definition \ref{defi_set_of_atoms}). 
This set has been studied in \cite{Popa_Shlya_Cartan_sub} and in \cite{Mukherjee_masa_decompo_integration}.
We identify $\Y$ as a symmetric relation on $Y$ and prove that it is quasi-invariant.
Furthermore, we show that $\WR$ is weakly contained in $\Y$.
Therefore, $\RT$ and $\WR$ are quasi-invariant.
We also prove that $\Y$ is equivalent to $\RT$ under the relation "$\equiv$", hence the main theorem of this paper is the following:
\begin{theo*}
Let $\AM$ be a MASA in a finite \vna, then the Takesaki equivalence relation, the \bieqre, the set of atoms and the \eqre\ induced by the normalizer are equal up to a null set, 
i.e. 
$$\RT\equiv\WR\equiv\Y\equiv\Ncal.$$
\end{theo*}
In particular a MASA is singular (its normalizer is equal to the unitaries of $A$) \ifff\ it is Takesaki simple.
We illustrate this result with a proposition on inclusions of countable groups, see proposition \ref{prop_groups}.
Furthermore, we deduce a result on the normalizer of tensor product of MASAs.
\begin{cor*}
Consider a family of MASAs in some \vna s $\{A_l\subset M_l,\,l\in\Lambda\}$, where $\Lambda$ is a countable set.
For any $l\in \Lambda$, we consider a faithful trace $\tau_l$ on $M_l$.
Let $A=\bigotimes_l A_l$ and $M=\bigotimes_l M_l$ be the tensor products of those \vna s \wrt\ the traces $\tau_l$.
The \vna\, generated by the normalizer $\NMA$ is equal to the tensor product of the \vna s generated by the normalizers $N_{M_l}(A_l)$, i.e.
$$\NMA''=\bigotimes_{l\in\Lambda}N_{M_l}(A_l)''.$$
\end{cor*}
This result has been proved by Chifan \cite{chifan_normalizer_masa} using analytic technics.

The rest of the paper is organized into 3 sections. 
In the first one we fix some notations and review some basic facts about \eqre s.
We define the \eqre\ induced by the action of the normalizer on the algebra $A$, the \bieqre\ and the set of atoms. 
In the second section, we prove the main result of this paper.
In the third one we illustrate the main result in the context of representations of discrete countable groups. 
Then we prove the corollary on tensor product of MASAs.

\section{Notations and definitions}\label{section_notation}

In this section, we fix some notations and define the \bieqre\ $\WR$, the \eqre\ $\Ncal$ induced by the action of the normalizer and the set of atoms $\Y$.
Consider the $C^*$-algebra of continuous functions $\mathcal C(Y)$. 
Let $B\subset M$ be a $C^*$-subalgebra which is separable and weakly dense.
The equivalence class of $\RT$ does not depend on $B$. 
Therefore we can assume that $B$ contains the $C^*$-algebra $\CY$.
Hence we have the following square of inclusions:
$${\begin{array}{ccc}
  \Li & \subset & M \\
  \cup &  & \cup \\
  \mathcal C(Y) & \subset & B
\end{array}}.$$
\begin{defi}\label{defi_Tak_weak_Tak_eqre}
Consider the \bieqre\ $\WR$ which is defined such that $(s,t)\in \WR$ if the representation $\pi_t\vert_{\mathcal C(Y)}$ is unitarily equivalent to $\pi_s\vert_{\mathcal C(Y)}$.
\end{defi}
Let $\IY$ be the group of Borel automorphisms of $(Y,\nu)$ that preserve the class of the measure $\nu$. 
For any countable subgroup $H<\NMA$ we can define a group homomorphism $\Theta^H:H\longrightarrow \IY$ such that for any $u\in H$ and any $f\in A$, $u^*fu=f\circ \Theta_u^H$ $\nu$-almost everywhere ($\nu$-a.e.).
Consider the \oreqre\ 
$$\Ncal_H=\{(\Theta^H_u(t),t),\ t\in Y,\ u\in H\}.$$
If $u\in\NMA$ we denote by $\Theta_u$ a given automorphism such that $u^*fu=f\circ \Theta_u$ $\nu$-a.e..
The following proposition justifies the definition of $\Ncal$ given in the introduction:
\begin{prop}\label{prop_Ncal}
There exists a countable subgroup $G<\NMA$ such that $G''=\NMA''\subset M$.
We denote by $\Ncal$ the orbit \eqre\ $\Ncal_G$.
If $H<\NMA$ is a countable subgroup then $\Ncal_H\prec\Ncal$.
Furthermore, $\Ncal_H\equiv\Ncal$ \ifff\ 
$$\{H\cup A\}''=\NMA''.$$
\end{prop} 
Before proving this proposition we recall a useful lemma:
\begin{lemm}\label{lem_Tak_unit_ortho}
Consider a unitary $u\in\NMA$. 
Suppose that there exists a Borel subset $E\subset Y$ such that for any $t\in E$, $\Theta_u(t)\neq t$.
Then, $\tau(u\chi_E)=0$, where $\chi_E$ is the characteristic function of the set $E$.
\end{lemm}
\begin{proof}
Let us show that $E_A(u\chi_E)=0$.
Let $f\in L^\infty(Y,\nu)$ be an injective function.
We have that
\begin{align*}
fE_A(u\chi_E)&=E_A(fu\chi_E)=E_A(u(u^*fu)\chi_E)=E_A(u(f\circ\Theta_u)\chi_E)\\
&=E_A(u\chi_E)(f\circ\Theta_u)=(f\circ\Theta_u)E_A(u\chi_E).
\end{align*}
We identify $E_A(u\chi_E)$ with a function of the algebra $L^\infty(Y,\nu)$.
We have that $(f-f\circ\Theta_u)(t)E_A(u\chi_E)(t)=0$ \alev.
The function $f$ is injective and $\Theta_u(t)\neq t$ for any  $t\in E$.
Therefore, $E_A(u\chi_E)(t)=0$ \alev, hence $E_A(u\chi_E)=0$.
This implies that $\tau(u\chi_E)=\tau\circ E_A(u\chi_E)=0$.
\end{proof}
Let us prove the proposition:
\begin{proof}[Proof of the proposition \ref{prop_Ncal}]
Let $G<\NMA$ be a countable subgroup which is dense for the norm of $\LM$. This group satisfies that $G''=\NMA''$.
To prove the two other statements of the proposition it is sufficient to show that for any countable subgroups $H,K<\NMA$ we have that 
$$\{H\cup A\}''\subset \{K\cup A\}''$$
\ifff\ $\Ncal_H\prec\Ncal_K.$

Suppose that $\mathcal N_H\prec\mathcal N_K,$ let $u\in \{H\cup A\}''$ be a unitary.
Let us show that $u\in \{K\cup A\}''$.
Let  $\{v_k,\,k\geqslant 1\}$ be an enumeration of the countable group $K$.
Consider the sets
$$E_k=\{t\in Y,\,\Theta^K_{v_k}(t)=\Theta_u(t)\}$$
and
$$F_k=E_k\backslash \bigcup_{j<k}E_j.$$
We see immediately that the sets $F_k$ are measurable.
Let $p_k$ be the projection equal to the characteristic function of the set $F_k$. 
If $k\neq l$, then $v_kp_k\perp v_lp_l$.
Therefore, the sum
$$\sum_{k=1}^\infty v_kp_k.$$ 
converges in the \vna\, $\{K\cup A\}''$ to an element $v$.
By hypothesis, the graph of $\Theta_u$ is weakly contained in $\Ncal_K$.
This implies that $\Y\backslash \bigcup_k F_k$ is a null set.
Thus, $v=\sum_k v_kp_k$ is a unitary in $\{K\cup A\}''$ and by construction $\Theta_u=\Theta_v$ $\nu$-\alev.
Therefore, $vu^*$ is a unitary of $A$, thus $u\in \{K\cup A\}''$.

Suppose that $\{H\cup A\}''\subset\{K\cup A\}''.$
Consider a unitary $u\in H$ and the set
$$E=\{t\in Y,\,(\Theta^H_u(t),t)\notin \Ncal_K\}.$$
The set $E$ is measurable. 
Let $v=up$, where $p=\chi_E$.
Consider a unitary $w\in K\cup U(A)$, where $U(A)$ is the unitary group of $A$.
Let us show that $v$ is orthogonal to $w$, i.e. $\tau(w^*v)=0$.
We have that $\Theta_{w^*u}=\Theta_w^{-1}\circ\Theta^H_u$ a.e..
By assumption, for any $t\in E$, $\Theta^H_u(t)\neq\Theta_w(t)$ a.e., hence $\Theta_{w^*u}(t)\neq t$ a.e..
We can apply lemma \ref{lem_Tak_unit_ortho}; thus, $\tau(w^*v)=0$.
Therefore, the partial isometry $v$ is orthogonal to the \vna\ $\{K\cup A\}''$.
This implies that $p=0$ and so the graph of $\Theta^H_u$ is weakly contained in the \eqre\ $\Ncal_K$.
So, $\Ncal_H\prec\Ncal_K$.
\end{proof}
We define \textit{the set of atoms}.
Let $\mathcal A=\{\pi(A),\rho(A)\}''\subset \mathbb B(\LM)$ be the abelian von Neumann subalgebra generated by the left and right actions of $A$ on $\LM$.
Consider the coordinate projection $p:Y^2\longrightarrow Y$, $p(s,t)=t$ and the flip $\theta:Y^2\longrightarrow Y^2$, $\theta(s,t)=(t,s)$.
Following the proof of \cite[Theorem 1]{Feldman_Moore_Equv_Relations_Cohomology_vna_II}, there exists a Borel probability measure $\mu$ on $Y^2$ such that the \vna\ $\mathcal A$ is isomorphic to $L^\infty(Y^2,\mu)$. 
It is easy to see that $\mu$ is quasi-invariant \wrt\ the flip, i.e. $\theta_*\mu\approx\mu$.
Furthermore, the push-forward measure $p_*\mu$ is in the equivalence class of the measure $\nu$, therefore by \cite[Chap. 6,\S 3]{bourbaki_integration} there exists a \textit{disintegration of $\mu$ \wrt\ $(p,\nu)$.}
It means that there exists a unique a.e. family $\{\mu_t,\,t\in Y\}$ of probability measures on $Y$, such that for any positive measurable function $f:Y^2\longrightarrow \R_+$, the map 
$$t\mapsto \int_Y f(s,t)d\mu_t(s)$$ 
is measurable, and
$$\mu(f)=\int_Y \int_Yf(s,t)d\mu_t(s)d\nu(t).$$
\begin{defi}\label{defi_set_of_atoms}
The \textit{set of atoms} of the MASA $\AM$ is the set 
$$\Y=\{(s,t)\in Y^2,\,\mu_t(\{s\})>0,\,\mu_s(\{t\})>0 \}.$$
\end{defi}
The set $\Y$ defines a symmetric relation on $Y$, hence we call orbit of $t$ the set of $s\in Y$ such that $(s,t)\in\Y$.
Note that $\Y$ is a measurable subset of $Y^2$, see \cite[Proposition 3.3]{Mukherjee_masa_decompo_integration} for a proof.
\section{The main result}\label{section_main_result}

\begin{theo}\label{theo_Tak_main_result}
Consider the equivalence relations $\RT$, $\WR$, $\Ncal$ and the set of atoms $\Y$.
Then,
$$\RT\equiv\WR\equiv\Y\equiv\Ncal.$$
\end{theo}

\begin{proof}
By definition, $\RT\subset\WR$.
Let us show that $\WR\prec\Y$.
Consider a continuous function $f\in \CY$.\\
\textit{Claim:} The scalar $f(t)$ is an eigenvalue of the operator $\pi_t(f)$ $\nu$-\alev.
Proof of the claim: 
The inclusion $\AM$ gives us an inclusion of $A$-bimodules $_AL^\infty(Y,\nu)_A\subset{_A\LM_A}.$
As a right $A$-module, 
$$L^\infty(Y,\nu)_A=\int_Y^\oplus \C_t d\nu(t),$$
where $\C_t$ is the complex vector space of dimension one and $\pi_t(f)$ acts by multiplication by $f(t)$ on it $\nu$-\alev.
Therefore, $f(t)$ is an eigenvalue of $\pi_t(f)$ $\nu$-\alev.

The \vna\ $\mathcal A$ is isomorphic to $L^\infty(Y^2,\mu)$ and it acts on the \Hs\ $\LM$.
Hence, there exists a measurable field of \Hs s $\hst$ over $(Y^2,\mu)$ such that we have an isomorphism of $\mathcal A$-modules:
$$\phi :{\LM}\simeq\int_{Y^2}^\oplus \hst d\mu(s,t).$$
We define the following direct integral of \Hs s 
$$\mathcal K_t=\int_Y^\oplus \hst d\mu_t(s).$$
By a result of Guichardet \cite[Proposition 1]{guichardet_avn_discrete}, we have an isomorphism of right $A$-modules:
$$\LM_A\simeq\int_Y^\oplus\mathcal K_t d\nu(t).$$
Consider a continuous and injective function $f\in \CY$.
For any $t\in Y$, we associate to $f$ an operator $f_t\in \mathbb B(\mathcal K_t)$ determined by 
$$f_t\xi=\int^\oplus_Yf(s) \xi_s d\mu_t(s),$$
for any $$\xi=\int^\oplus_Y \xi_s d\mu_t(s)\in\int^\oplus_Y \hst d\mu_t(s).$$
We remark that 
$$\pi(f)=\int^\oplus_Yf_t d\nu(t).$$
By uniqueness of the disintegration there exists a null set $N\subset Y$ such that $\pi_t(f)=f_t$ for any $t\in Y\backslash N$.
The claim implies that there exists a null set $N_0$ such that for any $t\in Y\backslash N_0$ we have that $f(t)$ is an eigenvalue of $\pi_t(f)$.
Let $(s_0,t)\in \WR\backslash (N\cup N_0)^2$, hence we have that $f(s_0)$ is an eigenvalue of the operator $f_t$.
Then there exists a non null vector
$$\eta=\intI \eta_sd\mu_t(s)\in \K_t$$
such that $f_t(\eta)=f(s_0)\eta$, meaning that $(f(s)-f(s_0))\eta_{s}=0$ $\mu_t$-\alev.
This implies that $s_0$ is an atom of $\mu_t$  because $f$ is injective. By exchanging the role of $s_0$ and $t$ we get that $(s_0,t)\in \Y$.
Therefore, $\WR\prec \Y$.

Let us show that $\Y\prec\Ncal.$
\textit{Claim:}Let $X\subset \Ymus$ be a measurable subset.
The following assertions are equivalent:
\begin{enumerate}
  \item $X$ is a null set for $\mu$;
  \item $p_1(X)$ is a null set for $\nu$;
  \item $p_2(X)$ is a null set for $\nu$.
\end{enumerate}
Proof of the claim:
We have that
$$\mu(X)=\int_Y\mu_t(\{s,\ (s,t)\in X\})d\nu(t)=\int_{p_2(X)}\mu_t(\{s,\ (s,t)\in X\}  )d\nu(t).$$
The set $X$ is contained in $\Ymus$, hence for any $t\in p_2(X)$, $\mu_t(\{s,\ (s,t)\in X\}  )$ is strictly positif.
Therefore $\mu(X)=0$ if and only if $p_2(X)$ is a null set.
We know that the class of the measure $\mu$ is invariant under the flip; thus, $\mu(X)=0$ if and only if $\mu(\theta(X))=0$ if and only if $p_1(X)$ is a null set.\\
Let $h:N\subset Y\longrightarrow Y$ be a measurable map defined on a measurable subset $N\subset Y$ such that its graph $\G_h$ is weakly contained in $ \Y$.
We have that $p_1(\G_h)=h(N)$ and $p_2(\G_h)=N$.
Thus by the claim we have that $\nu(N)=0$ if and only if $h(N)$ is a null set.

The set of atoms is measurable, hence there exists a $\mu$-null set $N\subset Y^2$ such that $\Y\backslash N$ is a Borel set.
By \cite[Theorem 18.10]{Kechris_bouquin} there exists a countable family of  Borel automorphisms $h_k\in \IY$ such that 
$$\Y\backslash N= \bigcup_k\Gamma_{h_k}.$$
The claim implies that there exists a null set $N_0\subset Y$ such that $N\subset N_0^2$. 
Therefore 
$\Y$ is equivalent to the union of graphs $\bigcup_k\Gamma_{h_k}.$
To conclude that $\Y\prec \Ncal$, we need to show that for any $h_k$ there exists a unitary $u_k\in\NMA$ such that $\Theta_{u_k}=h_k$ $\nu$-\alev.
Hence, by considering the subgroup of $\NMA$ generated by the $u_k$ we will get that $\Y\prec\Ncal$.
\begin{lemm}\label{lem_Tak_graph c N}
Consider a Borel automorphism $h\in \IY$ such that its graph $\G_h$ is weakly contained in $\Y$.
Then, there exists a unitary in the normalizer of $A$, $u\in N_M(A)$, such that $\Theta_u= h$ $\nu$-\alev.
\end{lemm}
\begin{proof}
Let us write $L^2(M)$ as a direct integral of \Hs s over the \mesp\ $(Y^2,\mu)$:
$$\psi:\LM\longrightarrow\intII \hst d\mu(s,t).$$
Consider a vector
$$\xi^0=\intII \xist^0 d\mu(s,t)$$
such that $\Vert \xist^0\Vert_{\hst}=1$ $\mu$-\alev.
Then the $A$-bimodule generated by the vector $\xi^0$ is giving us an embedding of bimodules $\Ldd\subset \LM$.
Let $\xi=\chi_{\G_h}\in\Lii\subset\LM$ be the characteristic function of the graph of $h$.
The preceding claim implies that $\mu(\G_h)$ is strictly positif.
It is easy to see that $\Vert \xi\Vert_2^2=\mu(\G_h)$, hence $\xi$ is different from zero.
Furthermore, for any $f\in A$, 
\begin{equation}\label{equa_entrelacement}
f.\xi=\xi.f\circ h^{-1}.
\end{equation}
The vector $\xi$ is an affiliated operator to $M$, therefore it admits a spectral decomposition $\xi=u\vert\xi\vert$, where $u$ is a partial isometry of $M$ and $\vert\xi\vert\in\LM_+$.
Let $p$ be the support of the partial isometry $u$. It is a projection of $A$, hence there exists a measurable subset $E\subset Y$ such that $1-p=\chi_E$.
We have that 
$$\Vert q.\xi\Vert^2=\int^\oplus_{Y^2} \chi_E(s)\chi_{\G_h}(s,t)d\mu(s,t)=\mu(\{(s,h(s)),\ s\in E\}).$$
By the preceding claim, $\mu(\{(s,h(s)),\ s\in E\})=0$ \ifff\ $\nu(E)=0$.
Therefore, $1-p=0$.
This implies that $u$ is a unitary.
We have that $ufu^*=f\circ h$ for any $f\in A$.
\end{proof}
Therefore, we have that $\Y\prec\Ncal$.

Let us show that $\Ncal\prec \RT.$
Let $G<\NMA$ be a countable subgroup of $\NMA$ that implements the orbit equivalence relation $\Ncal$.
Let $u\in G$, by \cite[Theorem 1.2]{takesaki_invariant_masa}, the graph of $\Theta_u$ is weakly contained in $\mathcal R.$
Therefore, $\Ncal\prec \mathcal R$.
This achieves the proof of the theorem.
\end{proof}

\begin{cor}
A MASA in a finite \vna\ is singular \ifff\ it is Takesaki simple.
\end{cor}
\begin{proof}
This is a direct consequence of the fact that $\RT\equiv\Ncal$.
\end{proof}

\begin{rem}
Consider a $A$-bimodule $_A\h_A$.
We can define an \eqre\ as follows.
Let $\{\K_t,\ t\in Y\}$ be a measurable field of \Hs s such that the right $A$-module $\h_A$ is isomorphic to the direct integral 
$$\int^\oplus_Y \K_td\nu(t).$$
Let $\lambda:A\longrightarrow \mathbb B(\h)$ be the left action of $A$.
Consider the measurable field of representations of $\CY$, $\{\lambda_t,\ t\in Y\}$, such that 
$$\lambda=\int^\oplus_Y \lambda_td\nu(t).$$
Let $\WR_\h$ be the \eqre\ $\{(s,t)\in Y^2,\, \lambda_t\simeq\lambda_s\}$.
If $\h=\LM$ as $A$-bimodules, we get that $\WR_\h$ is the bimodule equivalence relation $\WR$.
Therefore, the bimodule \eqre\ can be reconstructed from $_A\LM_A$, and so does the \eqre\ $\Ncal$ by the theorem \ref{theo_Tak_main_result}.
Feldman and Moore \cite{Feldman_Moore_Equv_Relations_Cohomology_vna_I,Feldman_Moore_Equv_Relations_Cohomology_vna_II} showed that a Cartan subalgebra is characterized by the \eqre\ $\Ncal$ and a $2$-cocycle.
Hence, a Cartan subalgebra is characterized by the $A$-bimodule $\LM$ and a $2$-cocycle.
\end{rem}

\section{Illustration and consequences of the main theorem}

\subsection{Group \vna s}
Let $H<G$ be an inclusion of discrete countable groups such that $H$ is abelian and 
\begin{equation}\label{equa_relative_ICC}
\text{for any}\ g\in G\backslash H\ \text{the set}\ \{hgh^{-1},\ h\in H\}\ \text{is infinite}.
\end{equation}
This implies that the inclusion of group \vna s $L(H)\subset L(G)$ is a MASA in a finite \vna\ (see Godement \cite{Godement_trans_fourier_groupes_discrets}).
Let $C\subset G$ be a system of left coset representatives of $H$.
We denote the unity of $G$ by $1$ and assume that $1\in C$.
Consider two functions $\sigma:G\rightarrow C$ and $\eta:G\rightarrow H$ such that for any $g\in G$ we have $g=\sigma(g)\eta(g).$
If $I$ is a set we denote the Hilbert space of square summable complex valued functions on $I$ by $\ell^2(I)$.
We denote the standard basis of $\ell^2(G)$ (resp. $\ell^2(H)$, resp. $\ell^2(C)$) by $\{e_g,\,g\in G\}$ (resp. $\{e_h,\,h\in H\}$ resp. $\{\varepsilon_c,\,c\in C\}$).
Let $\pi,\rho:G\curvearrowright \ell^2(G)$ be the left and right regular representations of the group $G$.\\
Let us decompose $\pi$ \wrt\ the right action of $H$.
Consider the dual group of $H$ with its Haar measure $(\widehat H,\nu)$ and the Fourier transform $\mathcal F:\ell^2(H)\longrightarrow L^2(\widehat H,\nu)$ defined such that
$\mathcal F(e_h)(t)=t(h)$ for any $h\in H$ and any character $t\in \widehat H$.
We have an isomorphism of right $L(H)$-modules
$\phi:\ell^2(G)\longrightarrow \ell^2(C)\otimes L^2(\widehat H,\nu)$ given by the formula 
$$\phi(e_g)(t)= t(\eta(g))\varepsilon_{\sigma(g)}.$$
Consider the representation $\pi_t:G\curvearrowright \ell^2(C)$ defined such that 
$$\pi_t(g)\varepsilon_c=t(\eta(gc))\varepsilon_{\sigma(gc)},$$
for any $t\in \widehat H$, $g\in G$ and $c\in C$.
An easy computation shows that this gives a disintegration of the representation $\pi$ \wrt\ the right action of $H$, i.e.
$$\pi=\int_{\widehat H} \pi_td\nu(t).$$

\begin{prop}\label{prop_groups}
If $s,t$ are two characters of $H$, then the two representations $\pi_s$ and $\pi_t$ are unitarily equivalent \ifff\ their restrictions to the abelian subgroup $H$ are unitarily equivalent.
Denote by $\AM$ the inclusion of \vna s $L(H)\subset L(G)$.
The group normalizer $\NGH$ generates the same \vna\ than the normalizer of the algebras $\NMA$.
\end{prop}
\begin{proof}
We recognize the Takesaki \eqre\ and the bimodule \eqre\ of $L(H)\subset L(G)$ which are respectively $\RT=\{(s,t)\in\wH^2,\ \pi_s\simeq\pi_t\}$  and $\WR=\{(s,t)\in\wH^2,\ \pi_s\vert_H\simeq\pi_t\vert_H\}$.
Consider the \oreqre\ $\Ncal_0$ given by the action $ad:\NGH\curvearrowright\wH$.
If $\RT=\WR$, then we have the first assertion of the proposition.
If $\Ncal_0\equiv \WR$, then the proposition \ref{prop_Ncal} implies the second assertion of the proposition.
Let us show that $\RT=\WR=\Ncal_0$.
By definition, $\mathcal R\subset \WR$.
Let us show that $\WR\subset \Ncal_0$.
Let $(s,t)\in \WR$, and $v$ a unitary of $\ell^2(C)$ such that $v^*\pi_s(h) v=\pi_t(h)$ for any $h\in H$.
Consider $h\in H$,
$$ \pi_s (h) v(\varepsilon_1) =v \pi_t (h) (\varepsilon_1) =v ( t(h) \varepsilon_1 ) = t(h) v(\varepsilon_1).$$
We denote by $y=\sum_{c\in C}y_c\varepsilon_c$ the vector $v(\varepsilon_1)\in \ell^2(C)$, decomposed in the orthonormal basis $\{\varepsilon_c,\ c\in C\}$.
Thus, for any $h\in H$ and any $c\in C$,
\begin{equation}\label{egalitey}
     y_c s(\eta(hc))=y_{\sigma(hc)}t (h).
\end{equation}
Let $c\in C$ such that $y_c\neq 0$.
The last identity tell us that for any $h$ in $H$, $\vert y_{\sigma(hc)}\vert=\vert y_c\vert.$
Therefore $p(HcH)$ is finite, where $p:G\twoheadrightarrow G/H$ is the canonical projection on the set of the right cosets.
The condition \ref{equa_relative_ICC} implies that $c$ is in the normalizer $\NGH$.
The equation \ref{egalitey} implies that $s=ad_c(t)$.
Hence $\WR\subset \Ncal_0$.

Let us show that $\Ncal_0\subset \RT$. 
Consider $(s,t)\in \Ncal_0$ and $g\in\NGH$ such that $s=ad_g(t)$.
Let $u$ be the unitary of $\ell^2(C)$ defined as follows:
$$u(\varepsilon_c)=s(\eta(cg))\varepsilon_{\sigma(cg)},$$
where $c\in C$.
An easy computation shows that $u\pi_tu^*=\pi_s$, hence $\Ncal_0\subset\RT$.
\end{proof}



\subsection{Tensor product of MASAs}\label{section_Tak tensor product}

In this section we compute the set of atoms for MASAs constructed from tensor product of MASAs.
We deduce a result on the tensor product of \vna s generated by the normalizers.

Let $\Lambda$ be a countable set and $\{A_l\subset M_l,\,l\in\Lambda\}$ a family of MASAs in some finite \vna s. We fix a trace $\tau_l$ on each $M_l$.
Consider the infinite tensor products of \vna s \wrt\ the traces $\tau_l$
$$\bigotimes_{l\in\Lambda}A_l\subset \bigotimes_{l\in\Lambda} M_l.$$
We denote those \vna s by $A=\bigotimes_l A_l$ and by $M=\bigotimes_l M_l$.
The Tomita commutant theorem implies that $\AM$ is a MASA.

For any $l\in\Lambda$, we consider a \sps\ $(Y_l,\nu_l)$ and identify $A_l$ with $L^\infty(Y_l,\nu_l)$.
Let $L^2(M_l)$ be the GNS \Hs\ and $\pi_l,\rho_l:M_l\longrightarrow\mathbb B( L^2(M_l))$ the left and right actions of $M_l$.
Consider the \vna\ $\Al=\{\pi_l(A_l)\cup\rho_l(A_l)\}''\subset\mathbb B(\LMl)$.
Let $\mu_l$ be a measure on $Y_l^2$ such that $\mu_l(\Delta Y_l)=1$ and such that $\Al\simeq L^\infty(Y_l^2,\mu_l)$, where $\Delta Y_l$ is the diagonal.
Consider a disintegration of $\mu_l$ \wrt\ the projection $p_l(s,t)=t$:
$$\mu_l=\int_{Y_l} \mu_{l,t} d\nu_l(t).$$
We denote the set of atoms of $\AMl$ by $\Y_l$.
Let $Y$ be the cartesian product of the $Y_l$ with the $\sigma$-algebra generated by the subsets of $Y$ of the form $\prod_l X_l$, where $X_l\subset Y_l$ is a measurable subset that is equal to $Y_l$ for all but a finite number of $l\in\Lambda$.
Let $\nu=\otimes_l \nu_l$ be the unique probability measure on $Y$ that satisfies $\nu(\prod_l X_l)=\prod_l\nu_l(X_l)$.
We identify $A$ with $L^\infty(Y,\nu)$.
We denote an element of $Y$ by $\underline t$ and its $l$-component by $t_l$.

\begin{theo}\label{theo_Tak_tensor}
The set of atoms $\Y$ of $\AM$ is the set of couples $(\underline s,\underline t)$ such that for any $l$ $(s_l,t_l)\in\Y_l$ and $s_l=t_l$ for all but a finite number of $l\in\Lambda$.
\end{theo}

\begin{proof}
Let $\Omega^l$ be the image of the unity of $M_l$ in the \Hs\ $\LMl$.
Consider the infinite tensor product of \Hs s $\bigotimesl \LMl$ \wrt\ the vectors $\Omega^l$.
There is a unitary transformation between $\LM$ and $\bigotimesl \LMl$ that conjugates the actions $\pi,\rho$ with the tensor product of actions $\otimesl \pi_l$ and $\otimesl \rho_l$.
For any $l\in\Lambda$, there exists a measurable field of \Hs s $\{\hst^l,\, s,t\in Y_l\}$ such that the $A_l$-bimodule $\LMl$ is isomorphic to the direct integral
$$\int^\oplus_{Y_l^2}\hst^ld\mu_l(s,t).$$
Consider the disintegration of the vector $\Omega^l$ which is
$$\int^\oplus_{Y_l^2}\Omega_{s,t}^ld\mu_l(s,t).$$
For any couple $(\underline s,\underline t)\in Y^2$, we consider the infinite tensor product of \Hs s $\h_{\underline s,\underline t}=\bigotimesl \h_{s_l,t_l}^l$ \wrt\ the vectors $\Omega_{s_l,t_l}^l$.
For any finite subset $E\subset \Lambda$, there exists a unique measure $\mu_E$ on $Y^2$ that satisfies:
$$\mu_E(\prod_{l\in\Lambda} X_l)=\prod_{l\in E}\mu_l(X_l)\times\prod_{l\in \Lambda\backslash E}\mu_l(X_l\cap \Delta Y_l).$$
Consider the measure class $\mathcal C$ of $Y^2$ generated by all the measures $\mu_E$.
Let $\mu$ be an element of the class $\mathcal C$.
It is clear that the $A$-bimodule $\LM$ is isomorphic to the direct integal of \Hs s $\{\h_{\underline s,\underline t},\, (\underline s,\underline t)\in Y^2\}$ over the measure $\mu$.
It means that 
$$\LM \simeq \int^\oplus_{Y^2} \h_{\underline s,\underline t}d\mu(\underline s,\underline t).$$
Consider the projection of $Y^2$ $p(\underline s,\underline t)=\underline t$.
Let 
$$\mu=\int_Y \mu_{\underline t} d\nu(\underline t)$$
be a disintegration of $\mu$ \wrt\ $p$ and let $\Y$ be the set of atoms of $\AM$.
By the construction of $\mu$, it is clear that $\Y$ is equal to the set of couples $(\underline s,\underline t)$ such that for any $l$ $(s_l,t_l)\in\Y_l$ and $s_l=t_l$ for all but a finite number of $l$.\end{proof}

\begin{cor}\label{coro_Tak_prod_tens_normalizer}
Let $N_{M_l}(A_l)$ be the group normalizer of the MASA $A_l\subset M_l$, and consider the \vna\, generated by those normalizers
$$\bigotimesl N_{M_l}(A_l)''\subset \bigotimesl \LMl.$$
Then, $$\NMA''=\bigotimesl N_{M_l}(A_l)''.$$
In particular, a tensor product of singular MASAs is a singular MASA.
\end{cor}

\begin{proof}
For any $l\in\Lambda$ consider a countable subgroup $G_l<\NMAl$ such that $G_l''=\NMAl''\subset M_l$.
By proposition \ref{prop_Ncal} and theorem \ref{theo_Tak_main_result} we have that the orbit \eqre\ $\Ncal_{G_l}$ is equivalent to $\Y_l$.
Let $G$ be the countable subgroup of $\NMA$ generated by elements of the form $u=\otimesl u_l$, where $u_l\in G_l$ and $u_l=1$ for all but finitely many $l\in\Lambda$.
We see immediately that the orbit \eqre\ induced by $G$ is equal to the set of couples $(\underline s,\underline t)$ such that for any $l\in\Lambda$, $(s_l,t_l)\in\Ncal_{G_l}$ and $s_l=t_l$ for all but finitely many $l\in\Lambda$.
Therefore the theorem \ref{theo_Tak_main_result} implies that $\Ncal_G=\Y$.
Hence, by proposition \ref{prop_Ncal} we have that $G''=\NMA''\subset M$.
By construction 
$$G''=\bigotimesl G_l''=\bigotimesl\NMAl''\subset\bigotimesl\LMl.$$
Therefore, $$\NMA''=\bigotimesl \NMAl''.$$
\end{proof}



\end{document}